\documentclass{amsart}

\usepackage{amsmath}
\usepackage{amssymb}

\usepackage[matrix,arrow,arc]{xy}

\usepackage{enumerate}

\renewcommand{\mod}{\operatorname{mod}}
\newcommand{\Hom}{\operatorname{Hom}}
\newcommand{\End}{\operatorname{End}}
\newcommand{\Ext}{\operatorname{Ext}}
\newcommand{\add}{\operatorname{add}}
\newcommand{\ann}{\operatorname{ann}}
\newcommand{\ind}{\operatorname{ind}}
\newcommand{\id}{\operatorname{id}}
\newcommand{\pd}{\operatorname{pd}}
\newcommand{\soc}{\operatorname{soc}}
\newcommand{\rad}{\operatorname{rad}}

\newcommand{\Tr}{\operatorname{Tr}}
\newcommand{\T}{\operatorname{T}}
\newcommand{\umod}{\operatorname{\underline{mod}}}

\newcommand{\op}{\operatorname{op}}
\newcommand{\Coker}{\operatorname{Coker}}
\newcommand{\res}{\operatorname{res}}
\newcommand{\rk}{\operatorname{rk}}

\newcommand{\cC}{\mathcal{C}}
\newcommand{\cD}{\mathcal{D}}

\newtheorem{theorem}{Theorem}[section] 
\newtheorem{lemma}[theorem]{Lemma}   

\newtheorem{proposition}[theorem]{Proposition}

\theoremstyle{definition}
\newtheorem*{example}{Example}

\title[Self-injective algebras with hereditary stable slice]%
 {Self-injective algebras with hereditary stable slice}

\author[A. Skowro\'nsk]{Andrzej Skowro\'nski}
\address[Andrzej Skowro\'nski]{%
   Faculty of Mathematics and Computer Science\\
   Nicolaus Copernicus University\\
   Chopina~12/18\\
   87-100 Toru\'n\\
   Poland}
\email{skowron@mat.uni.torun.pl}
\author[K. Yamagata]{Kunio Yamagata}
\address[Kunio Yamagata]{Department of Mathematics\\
    Tokyo University of Agriculture and Technology\\
    Nakacho 2-24-16, Koganei\\
    Tokyo 184-8588\\
   Japan}
\email{yamagata@cc.tuat.ac.jp}


\begin{document}
\maketitle

\begin{center}
\vspace*{-5mm}
\textit{Dedicated to Karin Erdmann on the occasion of her seventieth birthday}
\vspace*{5mm}
\end{center}

\begin{abstract}
We determine the structure of all finite-dimensional self-injective algebras
over a field whose Auslander-Reiten quiver admits a hereditary stable
slice.
\end{abstract}

\section{Introduction and the main result} 
\label{sec:intro}

\noindent

In this paper, by an algebra we mean a basic, indecomposable,
finite-dimensional associative $K$-algebra with identity over
a field $K$.
For an algebra $A$, we denote 
by $\mod A$ the category of finite-dimensional right $A$-modules, 
by $\ind A$ the full subcategory of $\mod A$ 
formed by the indecomposable modules,
by $D$ the standard duality $\Hom_K(-,K)$ on $\mod A$, 
by $\Gamma_A$ the Auslander-Reiten quiver of $A$,
and by $\tau_A$ and $\tau_A^{-1}$ the Auslander-Reiten 
translations $D \Tr$ and $\Tr D$, respectively.
An algebra $A$ is called \emph{self-injective}
if $A_A$ is injective, or equivalently,
the projective modules in $\mod A$ are injective.
If $A$ is a self-injective algebra, then the left socle of $A$
and the right socle of $A$ coincide, and we denote them
by $\soc (A)$.
Two self-injective algebras $A$ and $A'$ are said
to be \emph{socle equivalent} if the quotient algebras
$A/\soc (A)$ and $A' / \soc (A')$ are isomorphic.

In the representation theory of self-injective algebras a prominent role
is played by the self-injective algebras $A$ which admit Galois
coverings of the form $\widehat{B} \to \widehat{B}/G=A$,
where $\widehat{B}$ is the repetitive category of an algebra $B$
of finite global dimension and $G$ is an admissible group
of automorphisms of $\widehat{B}$.
Namely, frequently interesting self-injective algebras
are socle equivalent to such orbit algebras $\widehat{B}/G$
and we may reduce their representation theory to that
for the corresponding algebras of finite global dimension
occurring in $\widehat{B}$.
For example, for $K$ algebraically closed, this is the case for
self-injective algebras of polynomial growth (see \cite{S3,S4}),
the restricted enveloping algebras 
\cite{FS1}, 
or more generally
the 
tame Hopf algebras 
with
infinitesimal group schemes \cite{FS2},
in odd characteristic, as well as for the special biserial algebras
\cite{DS,PS}.
We also mention that for algebras $B$ of finite global dimension
the stable module category $\umod \widehat{B}$
is equivalent (as a triangulated category)
to the derived category $D^b(\mod B)$ of bounded
complexes in $\mod B$ \cite{Ha}.

Among the algebras of finite global dimension a prominent role
is played by the tilted algebras of hereditary algebras,
for which the representation theory is rather well understood
(see \cite{ANS,AS,Bo,HR,JMS2,K1,K2,L,R1,R2,S1} and \cite{SY11} 
for some basic results and characterizations).
This made it possible to understand the representation theory
of the orbit algebras $\widehat{B}/G$ of tilted algebras $B$, called
\emph{self-injective algebras of tilted type} 
(we refer to
\cite{ANS,BS1,BS2,EKS,ES,Ho,HW,JMS1,JPS,KS,S3,S4,SY3,SY4,SY8,SY10}
for some general results and applications).
In particular, it was shown that every admissible
group $G$ of the repetitive category $\widehat{B}$
of a tilted algebra $B$ 
is an infinite cyclic group generated by a strictly
positive automorphism of $\widehat{B}$.
In the series of articles \cite{SY1,SY2,SY3,SY5,SY6,SY7} 
we developed the theory of self-injective algebras with
deforming ideals and established necessary and sufficient
conditions for a self-injective algebra $A$ to be socle equivalent
to an orbit algebra $\widehat{B}/G$, for an algebra $B$
and an infinite cyclic group $G$ generated by a strictly
positive automorphism of $\widehat{B}$ being the composition
$\varphi \nu_{\widehat{B}}$ of the Nakayama automorphism
$\nu_{\widehat{B}}$ of $\widehat{B}$
and a positive automorphism $\varphi$ of $\widehat{B}$.

In this paper we concentrate on the question of when
a self-injective algebra $A$, and its module category $\mod A$,
can be recovered from a finite collection of modules in $\ind A$
satisfying some homological conditions.
We will show that it is possible when these indecomposable modules 
form a hereditary stable slice in the Auslander-Reiten quiver
$\Gamma_A$ of $A$.

We shall describe the main result of the paper.

Let $A$ be a self-injective algebra and $\Gamma_A^s$
the stable Auslander-Reiten quiver of $A$, obtained from
$\Gamma_A$  by removing the projective modules and the arrows 
attached to them. 
Following \cite{SY10}, 
a full valued subquiver $\Delta$ of $\Gamma_A$ is said to be a
\emph{stable slice} if the following conditions are satisfied:
\begin{enumerate}
 \item
  $\Delta$ is connected, acyclic, and without projective modules.
 \item
  For any valued arrow $V \xrightarrow{(a,a')} U$ in $\Gamma_A$
  with $U$ in $\Delta$ and $V$ non-projective,
  $V$ belongs to $\Delta$ or to $\tau_A\Delta$.
 \item
  For any valued arrow $U \xrightarrow{(b,b')} V$ in $\Gamma_A$
  with $U$ in $\Delta$ and $V$ non-projective,
  $V$ belongs to $\Delta$ or to $\tau_A^{-1}\Delta$.
\end{enumerate}
Assume now that $\Delta$ is a finite stable slice of $\Gamma_A$.
Then $\Delta$ is said to be \emph{right regular} 
if $\Delta$ does not contain the radical $\rad P$ 
of an indecomposable projective module $P$ in $\mod A$.
More generally, $\Delta$ is said to be \emph{almost right regular} 
if for any indecomposable projective module $P$ from $\mod A$
with $\rad P$ lying on $\Delta$, $\rad P$ is a sink of $\Delta$.
Finally, $\Delta$ is said to be \emph{hereditary} 
if the endomorphism algebra $H(\Delta) = \End_A(M(\Delta))$
of the direct sum $M(\Delta)$ of all modules lying on $\Delta$
is a hereditary algebra and its valued quiver $Q_{H(\Delta)}$
is the opposite quiver $\Delta^{\op}$ of $\Delta$.

The following theorem is the main result of this paper and
extends results established in \cite{SY3,SY10}
to a general case.

\begin{theorem}
\label{th:main}
Let $A$ be a self-injective algebra over a field $K$.
The following statements are equivalent.
\begin{enumerate}[\upshape (i)]
 \item
  $\Gamma_A$ admits a hereditary almost right regular stable slice.
 \item
  $A$ is socle equivalent to the orbit algebra 
  $\widehat{B}/(\varphi\nu_{\widehat{B}})$,
  where $B$ is a tilted algebra and 
  $\varphi$ is a positive automorphism of $\widehat{B}$.
\end{enumerate}
Moreover, if $K$ is algebraically closed, 
we may replace in {\upshape{(ii)}} ``socle equivalent'' by ``isomorphic''. 
\end{theorem}

We would like to stress that in general we cannot replace
in {\upshape{(ii)}} ``socle equivalent'' by ``isomorphic''. 
Namely, there exist fields $K$, with non-zero second
Hochschild cohomology group $H^2(K,K)$, and non-splittable
Hochschild extensions
\[
  0 \to D(H) \to \widetilde{H} \to H \to 0
\]
of hereditary algebras $H$ over $K$ such that $\widetilde{H}$
is a self-injective algebra socle equivalent but
non-isomorphic to the trivial extension algebra
$\T(H) = H \ltimes D(H) = \widehat{H} / (\nu_{\widehat{H}})$
(see \cite[Corollary~4.2 and Proposition~6.1]{SY1}).

We also mention that the class of self-injective
algebras occurring in the statement (ii)
is closed under stable, and hence derived,
equivalences (see \cite{PX,SY2,SY7} and \cite{Ric}).

\section{Orbit algebras of repetitive categories}
\label{sec:orbitAlgebras}

Let $B$ be an algebra and $1_{B}=e_{1}+\cdots +e_{n}$ 
a decomposition of the identity of $B$ into a sum 
of pairwise orthogonal primitive idempotents. 
We associate to $B$ a self-injective locally bounded 
$K$-category $\widehat B$, called the 
\emph{repetitive category} of $B$ (see~\cite{HW}). 
The objects of $\widehat B$ are $e_{m,i}$, $m\in{\mathbb{Z}}$,
$i\in \{1, \dots, n\}$, and the morphism spaces are defined as follows
\[
\widehat B(e_{m,i},e_{r,j})=\left\{\begin{array}{ll}
e_{j}Be_{i},    & r=m,\\
D(e_{i}Be_{j}),& r=m+1,\\
0,& \textrm{otherwise}.
\end{array} \right.
\]
Observe that 
$e_{j}Be_{i}=\Hom _{B}(e_{i}B,e_{j}B)$, $D(e_{i}Be_{j})=e_{j}D(B)e_{i}$ 
and
\[
  \bigoplus_{(m,i)\in{\mathbb{Z}\times \{1, \dots ,n\}}} \widehat B(e_{m,i},e_{r,j})
    =e_{j}B \oplus D(Be_{j}),
\]
for any $r\in{\mathbb{Z}}$ and $j\in\{1, \dots ,n\}$.
We denote by $\nu_{\widehat B}$ the \emph{Nakayama automorphism} 
of $\widehat B$ defined by
\[
\nu_{\widehat B}(e_{m,i})=e_{m+1,i} \quad \textrm{for all} \quad (m,i)\in \mathbb{Z}\times\{1, \dots ,n\}.
\]
An automorphism $\varphi$ of the $K$-category $\widehat B$ is said to be:
\begin{itemize}%
\renewcommand{\labelitemi}{$\bullet$}
 \item \emph{positive} if, for each pair $(m,i)\in{\mathbb{Z}\times \{1, \dots ,n\}}$, 
 we have $\varphi(e_{m,i})=e_{p,j}$ for some $p\geq m$ and some $j\in\{1, \dots ,n\}$;
 \item \emph{rigid} if, for each pair $(m,i)\in{\mathbb{Z}\times \{1, \dots ,n\}}$, 
 there exists $j\in\{1, \dots ,n\}$ such that $\varphi(e_{m,i})=e_{m,j}$;
 \item \emph{strictly positive} if it is positive but not rigid.
\end{itemize}
Then the automorphisms $\nu^{r}_{\widehat B}$, $r\geq 1$, 
are strictly positive automorphisms of $\widehat B$.

A group $G$ of automorphisms  of $\widehat B$ is said to be 
\emph{admissible} if $G$ acts freely on the set of objects of 
$\widehat B$ and has finitely many orbits.
Then, following P.~Gabriel \cite{G}, we may consider 
the orbit category $\widehat B/G$ of $\widehat B$ with respect 
to $G$ whose objects are the $G$-orbits of objects in $\widehat B$,
and the morphism spaces are given by
\[
\big(\widehat B/G\big)(a,b)=
\Big\lbrace (f_{y,x})\in{\prod_{(x,y)\in{a\times b}}} 
\widehat B(x,y)\hspace{2mm}|\hspace{2mm} gf_{y,x}
=f_{gy,gx}, \forall_{g\in{G}, (x,y)\in{a\times b}}\Big\rbrace
\]
for all objects $a,b$ of $\widehat B/G$. Since $\widehat B/G$ has
finitely many objects and the morphism spaces in $\widehat B/G$
are finite-dimensional, we have the associated finite-dimensional 
self-injective $K$-algebra $\bigoplus (\widehat B/G)$ which is the
direct sum of all morphism spaces in $\widehat B/G$, called the
\emph{orbit algebra} of $\widehat B$ with respect to $G$. We will
identify $\widehat B/G$ with $\bigoplus (\widehat B/G)$. For
example, for each positive integer $r$, the infinite cyclic group
$(\nu^{r}_{\widehat B})$ generated by the $r$-th power
$\nu^{r}_{\widehat B}$ of $\nu_{\widehat B}$ is an admissible
group of automorphisms of $\widehat B$, and we have the associated
self-injective orbit algebra
\[
T(B)^{(r)}=\widehat B/(\nu^{r}_{\widehat
B})=
  \begin{Bmatrix} 
   \begin{bmatrix}
        b_{1} & 0 & 0 &\ldots & 0 & 0 & 0 \\
        f_{2} & b_{2} & 0 &\ldots & 0 & 0 & 0 \\
        0 & f_{3} & b_{3} &\ldots & 0 & 0 & 0 \\
        \vdots &\vdots &\ddots &\ddots &\vdots &\vdots &\vdots \\
        \vdots &\vdots &\vdots &\ddots &\ddots &\vdots &\vdots \\
        0 & 0 & 0 &\ldots & f_{r-1} & b_{r-1} & 0 \\
        0 & 0 & 0 &\ldots & 0 & f_{1} & b_{1}
   \end{bmatrix}\\
    b_{1}, \ldots ,b_{r-1}\in{B}, f_{1}, \ldots ,f_{r-1}\in{D(B)}
  \end{Bmatrix} ,
\]
called the $r$-\emph{fold trivial extension algebra of B}. 
In particular, $T(B)^{(1)}\cong T(B)=B\ltimes D(B)$ is 
the \emph{trivial extension algebra} of $B$ by the injective cogenerator $D(B)$.

Let $A$ be a self-injective algebra. For a subset $X$ of $A$, 
we may consider the left annihilator 
$l_{A}(X)=\{a\in A\hspace{2mm}|\hspace{2mm}aX=0\}$ 
of $X$ in $A$ and the right annihilator
$r_{A}(X)=\{a\in A\hspace{2mm}|\hspace{2mm}Xa=0\}$ 
of $X$ in $A$. 
Then by a theorem due to T.~Nakayama 
(see~\cite[Theorem~IV.6.10]{SY9}) 
the annihilator operation $l_{A}$ induces a Galois correspondence 
from the lattice of right ideals of $A$ to the lattice of left ideals of $A$,
and $r_{A}$ is the inverse Galois correspondence to $l_{A}$. 
Let $I$ be an ideal of $A$, $B=A/I$, and $e$ an
idempotent of $A$ such that $e+I$ is the identity of $B$. 
We may assume that $1_{A}=e_{1}+\cdots +e_{r}$ with $e_{1},\ldots ,e_{r}$ 
pairwise orthogonal primitive idempotents of $A$, 
$e=e_{1}+\cdots +e_{n}$ for some $n\leq r$, and 
$\{e_{i}\hspace{2mm}|\hspace{2mm}1\leq i\leq n\}$ 
is the set of all idempotents in 
$\{e_{i}\hspace{2mm}|\hspace{2mm}1\leq i\leq r\}$ 
which are not in $I$. 
Then such an idempotent $e$ is uniquely determined by $I$ 
up to an inner automorphism of $A$, and is called a \emph{residual identity} of $B=A/I$. 
Observe also that $B\cong eAe/eIe$.

Let $A$ be a self-injective algebra,
$I$ an ideal of $A$,
$B = A/I$, 
$e$ a residual identity of $B$ and
assume that $r_A(I) = eI$. 
Then we have a canonical isomorphism of algebras 
$eAe/eIe\to A/I$ and $I$ can be considered as an 
$(eAe/eIe)$-$(eAe/eIe)$-bimodule. 
Following \cite{SY1},
we denote by $A[I]$ the direct sum of $K$-vector spaces 
$(eAe/eIe)\oplus I$ with the multiplication
\[
(b,x)\cdot (c,y)=(bc,by+xc+xy)
\]
for $b,c\in{eAe/eIe}$ and $x,y\in I$. 
Then $A[I]$ is a $K$-algebra with the identity $(e+eIe,1_{A}-e)$, 
and, by identifying $x\in{I}$ with $(0,x)\in{A[I]}$,
we may consider $I$ 
to be the
ideal of $A[I]$. 
Observe that $e=(e+eIe,0)$ is a residual identity 
of $A[I]/I=eAe/eIe\cong A/I$.
We also note that $\soc(A) \subseteq I$ and
$l_{e A e}(I) = e I e = r_{e A e}(I)$,
by \cite[Proposition~2.3]{SY1}.

The following theorem is a consequence of results
established in \cite{SY1} and \cite{SY3}.

\begin{theorem}
\label{th:2.1}
Let $A$ be a self-injective algebra, $I$ ideal of $A$,
$B = A/I$, $e$ a residual identity of $B$. 
Assume that $r_A(I) = eI$
and the valued quiver $Q_B$ of $B$ is acyclic.
Then the following statements hold.
\begin{enumerate}[\upshape (i)]
\item
    $A[I]$ is a self-injective algebra with the same 
    Nakayama permutation as $A$.
\item 
    $A$ and $A[I]$ are socle equivalent.
\item
    $A[I]$ is isomorphic to the orbit algebra
    $\widehat{B}/(\varphi \nu_{\widehat{B}})$, 
    for some positive automorphism $\varphi$ of
    $\widehat{B}$.
\end{enumerate}
Moreover, if $K$ is an algebraically closed field, 
we may replace in {\upshape{(ii)}} ``socle equivalent'' by ``isomorphic''. 
\end{theorem}

\begin{proof}
The statements (i), (ii) and the final part of the theorem
follow from \cite[Theorems 3.2 and 4.1]{SY1}.
The statement (iii) follows from \cite[Theorem~4.1]{SY3}.
\end{proof}

\section{Proof of the necessity part of Theorem~\ref{th:main}}

Let $A$ be a self-injective algebra over a field $K$, and
assume that $\Gamma_A$ admits a hereditary almost regular
stable slice $\Delta$.
Let $M$ be the direct sum of all indecomposable modules
in $\mod A$ lying on $\Delta$, 
$I$ the right annihilator
$r_A(M) = \{ a \in M \, | \, M a = 0 \}$ of $M$ in $\mod A$,
$B = A / I$, and $H = \End_B(M)$.
We note that $H = \End_A(M)$ and hence is a hereditary algebra.
Moreover, the valued quiver $Q_H$ of $H$ is the opposite quiver
$\Delta^{\op}$ of $\Delta$.
This implies that every non-zero non-isomorphism in $\mod A$
between two modules lying on $\Delta$ is a finite sum of compositions 
of irreducible homomorphisms in $\mod A$ corresponding
to valued arrows of $\Delta^{\op}$.

\begin{lemma}
\label{lem:3.1}
Let $X$ be an indecomposable module from $\Delta$.
Then the following statements hold.
\begin{enumerate}[\upshape (i)]
 \item 
  $X$ is an injective $B$-module, if $X$ is the radical
  of an indecomposable projective $A$-module.
 \item 
  $\tau_B^{-1} X = \tau_A^{-1} X$,
  if $X$ is not the radical of an indecomposable 
  projective $A$-module.
\end{enumerate}
\end{lemma}

\begin{proof}
(i)
Assume $X = \rad P$ for an indecomposable projective
$A$-module $P$.
Since $X$ lies on $\Delta$ and $P$ is not in $\Delta$,
$X$ is the largest right $B$-submodule
of the injective $A$-module $P$, and consequently
$X$ is an injective $B$-module, because $B$ is a quotient
algebra of $A$.

(ii)
Assume that $X$ is not the radical of an indecomposable 
projective $A$-module.
Let $Y = \tau_A^{-1} X$ and $f : P(Y) \to Y$
be a projective cover of $Y$ in $\mod A$.
Since $\Delta$ is an almost right regular stable
slice of $\Gamma_A$, we conclude that $f$ factors
through a module $M^r$ for some positive
integer $r$. But then $Y$ is a $B$-module, and
hence $\tau_B^{-1} X = \tau_A^{-1} X$.
Clearly, then $X$ is not an injective $B$-module. 
\end{proof}

\begin{lemma}
\label{lem:3.2}
The following statements hold.
\begin{enumerate}[\upshape (i)]
 \item 
  $\Hom_B(\tau_B^{-1} M, M) = 0$.
 \item 
  $\id_B M \leq 1$.
\end{enumerate}
\end{lemma}

\begin{proof}
(i)
It follows from Lemma~\ref{lem:3.1}
that there is an epimorphism of right $B$-modules
$g : M^s \to \tau_B^{-1} M$ for some positive integer $s$.
Then we conclude that  $\Hom_B(\tau_B^{-1} M, M) = 0$,
because $\End_A(M) = \End_B(M)$ is a hereditary algebra
whose valued quiver 
is the opposite quiver $\Delta^{\op}$ of $\Delta$.

(ii)
Since $M$ is a faithful $B$-module there is a monomorphism
of right $B$-modules
$B \to M^t$
for some positive integer $t$
(see \cite[Lemma~II.5.5]{SY9}),
and hence\linebreak $\Hom_B(\tau_B^{-1} M, B) = 0$, by (i).
This implies that $\id_B M \leq 1$ (see \cite[Proposition~III.5.4]{SY9}).
\end{proof}

\begin{lemma}
\label{lem:3.3}
The following statements hold.
\begin{enumerate}[\upshape (i)]
 \item
  For any valued arrow $U \xrightarrow{(c,c')} V$ in $\Gamma_B$
  with $U$ in $\Delta$,
  $V$ belongs to $\Delta$ or to $\tau_B^{-1}\Delta$.
 \item
  For any valued arrow $V \xrightarrow{(d,d')} U$ in $\Gamma_B$
  with $U$ in $\Delta$,
  $V$ belongs to $\Delta$ or to $\tau_B\Delta$.
\end{enumerate}
\end{lemma}

\begin{proof}
(i)
It follows from Lemma~\ref{lem:3.1}
and the fact that $\Delta$ is a stable slice in $\Gamma_A$.

(ii)
Assume that $V \xrightarrow{(d,d')} U$ is a valued arrow 
in $\Gamma_B$ with $U$ in $\Delta$.
We may assume that $V$ is not on $\Delta$.
We claim that $V$ is not an injective $B$-module.
Suppose
it is not the case.
Because $M$ is a faithful $B$-module, there is an epimorphism 
of right $B$-modules $M^p \to D(B)$
for some positive integer $p$, by the dual of
\cite[Lemma~II.5.5]{SY9},
because $l_A(D(M)) = r_A(M) = 0$.
Then there exist homomorphisms of right $B$-modules
$W \xrightarrow{f} V \xrightarrow{g} U$
with $g f \neq 0$ and $W$ an indecomposable 
$B$-module lying on $\Delta$.
Since $V$ is not in $\Delta$, this contradicts
to the fact that $\Delta^{\op}$ is the valued quiver
of the algebra $\End_B(M)$.
Therefore, $V$ is not injective in $\mod B$.
But then $\tau_B^{-1} V$ is an indecomposable
module and there is a valued arrow
$U \xrightarrow{(d',d)} \tau_B^{-1}V$ in $\Gamma_B$
(see \cite[Lemma~III.9.1 and Proposition~III.9.6]{SY9}).
Then it follows from (i)  that 
$\tau_B^{-1} V$ belongs to $\Delta$ or to $\tau_B^{-1}\Delta$.
Since $V$ is not in $\Delta$, we conclude that
$\tau_B^{-1} V$ belongs to  $\Delta$, and hence 
$V$ belongs to $\tau_B\Delta$.
\end{proof}

\begin{lemma}
\label{lem:3.4}
The following statements hold.
\begin{enumerate}[\upshape (i)]
 \item 
  $\Hom_B(M, \tau_B M) = 0$.
 \item 
  $\pd_B M \leq 1$.
\end{enumerate}
\end{lemma}

\begin{proof}
(i)
Consider an injective envelope 
$h : \tau_B M \to E(\tau_B M)$ 
of $\tau_B M$ in $\mod B$.
Since $\tau_B M$ has no injective direct summands,
it follows from Lemma~\ref{lem:3.3}
that $h$ factors through a module $M^q$ for some
positive integer $q$, and hence there is a monomorphism
of right $B$-modules
$u : \tau_B M \to M^q$ .
But then 
 $\Hom_B(M, \tau_B M) = 0$,
because $\Delta^{\op}$ is the valued quiver 
of the hereditary algebra $\End_B(M)$.

(ii)
Since $M$ is a faithful $B$-module there is an epimorphism
of right $B$-modules
$M^p \to D(B)$
for a positive integer $p$,
and consequently 
$\Hom_B(D(B), \tau_B M) = 0$.
This implies that $\pd_B M \leq 1$
(see \cite[Proposition~III.5.4]{SY9}).
\end{proof}

\begin{proposition}
\label{prop:3.5}
The following statements hold.
\begin{enumerate}[\upshape (i)]
 \item
  $M$ is a tilting $B$-module.
 \item
  $T = D(M)$ is a tilting module in $\mod H$.
 \item
  There is a canonical isomorphism of $K$-algebras
  $B \xrightarrow{\sim} \End_H(T)$.
 \item
  $\Delta$ is the section $\Delta_T$ of the connecting component 
  $\cC_T$ of $\Gamma_B$ determined by $T$.
\end{enumerate}
\end{proposition}

\begin{proof}
(i)
Let $f_1,\dots,f_d$ be a basis of the $K$-vector space
$\Hom_B(B,M)$.
Then we have a monomorphism 
$f : B \to M^d$ in $\mod B$,
induced by $f_1,\dots,f_d$,
and hence a short exact sequence
\[
  0 \to B \xrightarrow{f} M^d \xrightarrow{g} N \to 0
\]
in $\mod B$,
where $N = \Coker f$ and $g$ is a canonical epimorphism.
Then, applying standard arguments using $\pd_B M \leq 1$,
we conclude  (see the proof of \cite[Proposition~3.8]{SY10})
that $M \oplus N$ is a tilting $B$-module.
We prove now that $N$ belongs to the additive category
$\add(M)$ of $M$.
Assume to the contrary that there exists an indecomposable
direct summand $W$ of $N$ which does not belong to $\add(M)$,
or equivalently $W$ does not lie on $\Delta$.
Clearly, $\Hom_B(M,W) \neq 0$ because $W$ is a quotient
module of $M^d$.
Applying now Lemma~\ref{lem:3.3},
we conclude that $\Hom_B(\tau_B^{-1} M, W) \neq 0$.
Moreover, by Lemma~\ref{lem:3.2}, we have $\id_B M \leq 1$.
Then, applying \cite[Corollary~III.6.4]{SY9},
we infer that 
$\Ext_B^1(W, M) \cong D \Hom_B(\tau_B^{-1} M, W) \neq 0$,
which contradicts $\Ext_B^1(N, M)  = 0$.
Therefore, $M$ is a tilting module in $\mod B$.
We also note that the rank of $K_0(B)$
is the number of indecomposable modules
lying on $\Delta$.

(ii)--(iv)
It follows from tilting theory that $T = D(M)$
is a tilting module in $\mod H$ and there is a canonical
isomorphism of $K$-algebras
$B \xrightarrow{\sim} \End_H(T)$
(see \cite[Proposition~VIII.3.3]{SY11}).
In particular, $B$ is a tilted algebra of type $\Delta^{\op}$.
Moreover, we have isomorphisms of right $B$-modules
\[
  \Hom_H\big(T, D(H)\big) = 
  \Hom_H\big(D(M), D(H)\big) \cong 
  \Hom_{H^{\op}}\big(H, M\big) \cong M ,
\]
because $M$ is a right $H^{\op}$-module.
Therefore, $\Delta$ is the canonical section
$\Delta_T$ of the connecting component
$\cC_T$ of $\Gamma_B$ determined by $T$
(see \cite[Theorem~VIII.6.7]{SY11}).
\end{proof}

We may choose a set
$e_{1},\ldots ,e_{r}$ of pairwise orthogonal primitive 
idempotents of $A$ such that $1_{A}=e_{1}+\cdots +e_{r}$ and,
for some $n\leq r$, 
$\{ e_i \, | \, 1\leq i \leq n \}$ is the set of all idempotents in
$\{ e_i \, | \, 1\leq i \leq r \}$ which are not in $I$.
Then $e=e_{1}+\cdots +e_{n}$ is a \emph{residual identity} of $B=A/I$. 

Our next aim is to prove that $r_A(I) = e I$.

We denote by $J$ the trace ideal of $M$ in $A$, that is, the ideal of $A$ 
generated by the images of all homomorphisms from $M$ to $A$ in $\mod A$, 
and by $J'$ the trace ideal of the left $A$-module $D(M)$ in $A$. 
Observe that $I=l_{A}(D(M))$. 
Then we have the following lemma.

\begin{lemma}
\label{lem:3.6}
We have $J\subseteq I$ and $J'\subseteq I$.
\end{lemma}

\begin{proof}
First we show that $J\subseteq I$. 
By definition, there exists an epimorphism $\varphi : M^{s}\to J$ 
in $\mod A$ for some positive integer $s$. 
Suppose that $J$ is not contained in $I$. 
Then there exists a homomorphism $f : A\to M$ in $\mod A$ 
such that $f(J)\neq 0$. 
Then we have the sequence of homomorphisms in $\mod A$
\[
  M^s \xrightarrow{\varphi} J \xrightarrow{u} A \xrightarrow{f}  M
\]
with $u$ being the inclusion homomorphism.
But then $f( u \varphi) = f u \varphi \neq 0$ and this
contradicts the fact that $\Delta^{\op}$ is the valued
quiver of $H = \End_A(M)$.
Hence $J \subseteq I$.

Suppose now that $J'$ is not contained in $I$. 
Then there exists a homomorphism $f'\colon A\to D(M)$ in
$\mod A^{\op}$ such that $f'(J')\neq 0$. 
Moreover, we have in $\mod A^{\op}$ an epimorphism 
$\varphi'\colon D(M)^{m}\to J'$ for some positive integer $m$. 
Then $f'u'\varphi'\neq 0$ for the inclusion homomorphism 
$u'\colon J'\to A$.
Applying the duality functor $D : \mod A^{\op}\to\mod A$ 
we obtain homomorphisms in $\mod A$ 
\[
  D\big(D(M)\big) \xrightarrow{D(f')} D(A) \xrightarrow{D(u')} 
  D(J') \xrightarrow{D(\varphi')}  D\big(D(M)^{m}\big) ,
\]
where $D(D(M))\cong M$, $D(D(M)^{m})\cong M^{m}$, $D(A)\cong A$ in $\mod A$, 
and $D(\varphi')D(u')D(f')=D(f'u'\varphi')\neq 0$. 
This again contradicts the fact that $\Delta^{\op}$ is the valued
quiver of $H = \End_A(M)$.
\end{proof}

\begin{lemma}
\label{lem:3.7}
We have $l_{A}(I)=J$, $r_{A}(I)=J'$ and $I=r_{A}(J)=l_{A}(J')$.
\end{lemma}

\begin{proof}
See \cite[Lemma~5.10]{SY1} or  \cite[Lemma~3.10]{SY10}.
\end{proof}

\begin{lemma}
\label{lem:3.8}
We have $eIe=eJe$. 
In particular, $(eIe)^{2}=0$.
\end{lemma}

\begin{proof}
Since $B\xrightarrow{\sim} eAe/eIe$ canonically,
$M$ is a module in $\mod eAe$ with $r_{eAe}(M)=eIe$. 
We note also that $eJe$ is the trace ideal of $M$ in $eAe$, 
generated by the images of all homomorphisms 
from $M$ to $eAe$ in $\mod eAe$. 
It follows from Lemma~\ref{lem:3.6} that $eJe=eJ$ is an ideal 
of $eAe$ with $eJe\subseteq eIe\subseteq \rad eAe$.
Let $C=eAe/eJe$. 
Then $M$ is a sincere module in $\mod C$.
We will prove that $M$ is a faithful module in $\mod C$. 
Observe that then $eIe/eJe=r_{C}(M)=0$, 
and consequently $eIe=eJe$. 
Clearly, then $(eIe)^{2}=(eJe)(eIe)=0$, because $JI=0$.

We prove the claim in several steps.

(1)
Assume that $\rad e_i A$ lies on $\Delta$, for some $i \in \{1,\dots,n\}$.
Clearly, then $e_i = e_i e = e e_i$ and $e_i (\rad A) = e_i (\rad A) e$.
Moreover, we have $e_i (\rad A) = e_i J$,
because $e_i A$ does not lie on $\Delta$.
On the other hand, $e_i B = e_i A / e_i I$
is an indecomposable projective $B$-module.
Since $M$ is a faithful $B$-module, there exists a monomorphism
$e_i B \to M^r$ for some positive integer $r$.
Further, since $\Delta^{\op}$ is the valued quiver
of $\End_B(M)$, we conclude that 
the composed homomorphism
$e_i \rad A \hookrightarrow e_i A \twoheadrightarrow e_i B \to M^r$
is zero.
Hence $e_i I = e_i  (\rad A)$, and consequently $e_i J = e_i I$.
In particular, we conclude that $e_i (\rad A)$
is an injective module in $\mod C$.

(2)
Assume that $\rad e_i A$ lies on $\Delta$, for some $i \in \{n+1,\dots,r\}$.
Then $e_i A e \subseteq e_i (\rad A)$, and hence $e_i A e = e_i (\rad A)$,
because $e_i (\rad A)$ is a right $B$-module.
This shows that the canonical epimorphism
$e_i (\rad A) \to e_i (\rad A) / \soc(e_i A)$
is a minimal left almost split homomorphism in $\mod e A e$,
and hence $e_i (\rad A)$ is an injective $e A e$-module.
Clearly, then $e_i (\rad A)$ is also an injective $C$-module.

(3)
Assume that $X$ is a module on $\Delta$ which is not the radical
of an indecomposable projective module in $\mod A$.
Then it follows from Lemma~\ref{lem:3.1}
that there  is an almost split sequence in $\mod B$
\[
  0 \to X \to Y \to Z \to 0
\]
which is an almost split sequence in $\mod A$.
Recall that $B\xrightarrow{\sim} eAe/eIe$ canonically.
Applying now the properties of the restriction functor
$\res_e = (-) e : \mod A \to \mod e A e$
(see \cite[Theorem~I.6.8]{ASS}),
we conclude that the above sequence is an almost split
sequence in $\mod e A e$.
In particular, we conclude that 
$\tau_{e A e}^{-1} X = \tau_{B}^{-1} X$,
under the identification $B = e A e / e I e$.
Clearly, then we have also  $\tau_{C}^{-1} X = \tau_{B}^{-1} X$.

(4)
We may decompose $M = U \oplus V$ in $\mod B$
with $V$ being the direct sum of all indecomposable modules
on $\Delta$ which are not radicals of indecomposable
projective modules in $\mod A$.
It follows from (1) and (2) that $U$ is an injective $C$-module.
We prove now that $\id_C V \leq 1$.
We may assume that $V \neq 0$.
Observe that, by (3), we have 
$\tau_{e A e}^{-1} V = \tau_{C}^{-1} V = \tau_{B}^{-1} V$.
Consider the exact sequence
\[
0 \to eJe\xrightarrow{u} eAe\xrightarrow{v} C \to 0
\]
in $\mod C$, where $u$ is the inclusion homomorphism
and $v$ is the canonical epimorphism. 
Applying the functor $\Hom_{eAe}(\tau_{eAe}^{-1}V,-)\colon\mod eAe\to\mod K$ 
to this sequence, we get the exact sequence in $\mod K$ of the form
\begin{align*}
\Hom_{eAe}(\tau_{eAe}^{-1}V,eJe)
   &\xrightarrow{\alpha} \Hom_{eAe}(\tau_{eAe}^{-1}V,eAe)
   \\&
     \xrightarrow{\beta} \Hom_{eAe}(\tau_{eAe}^{-1}V,C)
     \xrightarrow{\gamma} \Ext_{eAe}^{1}(\tau_{eAe}^{-1}V,eJe) ,
\end{align*}
where $\alpha=\Hom_{eAe}(\tau_{eAe}^{-1}V,u)$, 
$\beta=\Hom_{eAe}(\tau_{eAe}^{-1}V,v)$, 
and $\gamma$ is the connecting homomorphism.
Since $\Delta$ is a section of the connecting component 
of $\Gamma_B$,
we conclude that there is an epimorphism $M^s \to \tau_B^{-1} V$
in $\mod B$, for some positive integer $s$,
and hence in $\mod e A e$.
Hence $\alpha$ is an isomorphism, because
$\tau_{e A e}^{-1} V = \tau_{B}^{-1} V$.
This implies that $\gamma$ is a monomorphism.
Further, every homomorphism from $M$
to $e J e$ in $\mod B$ factors through
$(\tau_{B}^{-1} V)^t = (\tau_{B}^{-1} M)^t$
for some positive integer $t$,
and hence there is an epimorphism
$(\tau_{B}^{-1} V)^t \to e J$.
Then we get $\Hom_{eAe}(eJe,V)=\Hom_{B}(eJe,V)=0$, because 
$\Hom_B(\tau_B^{-1} M, M) = 0$.
Then we obtain 
$\Ext_{eAe}^{1}(\tau_{eAe}^{-1}V,eJe)\cong D\overline{\Hom}_{eAe}(eJe,V)=0$. 
Summing up, we conclude that
$\Hom_{C}(\tau_{C}^{-1}V,C)=\Hom_{eAe}(\tau_{eAe}^{-1}V,C)=0$, 
and hence $\id_C V \leq 1$.

(5)
By (1), (2) and (4), we have $\id_C M \leq 1$.
Further, 
$\Ext_{C}^{1}(M,M) \cong\linebreak D\overline{\Hom}_{C}(\tau_{C}^{-1}M,M)
 = D\overline{\Hom}_{B}(\tau_{B}^{-1}M,M)=0$.
Since the rank of $K_{0}(C)$ is the rank of $K_{0}(B)$, 
which is the number of indecomposable direct summands of $M$,
we conclude that $M$ is a cotilting module in $\mod C$. 
Then $D(M)$ is a tilting module in $\mod C^{\op}$. 
In particular, $D(M)$ is a faithful module in $\mod C^{\op}$. 
Then we obtain the required fact 
$r_{C}(M)=r_{C^{\op}}(D(M))=0$.
\end{proof}

\begin{lemma}
\label{lem:3.9}
Let $f$ be a primitive idempotent in $I$ such that $fJ\neq fAe$.
Then $L=fAeAf+fJ+fAeAfAe+eAf+eIe$ is an ideal of $F=(e+f)A(e+f)$,
and $N=fAe/fLe$ is a module in $\mod B$ such that 
$\Hom_{B}(N,M)=0$ and $\Hom_{B}(M,N)\neq 0$.
\end{lemma}

\begin{proof}
It follows from Lemma~\ref{lem:3.8} that $fAeIe\subseteq fJ$. 
Then the fact that $L$ is an ideal of $F$ is a direct 
consequence of $fJ\subseteq fAe$.
Observe also that $N \neq 0$.
Indeed, if $fAe=fLe$ then
$fAe=fJ+fAe(\rad (eAe))$,
since $eAfAe\subseteq\rad (eAe)$.
But then $fAe=fJ$, by the Nakayama lemma 
\cite[Lemma~I.3.3]{SY9}, which contradicts our assumption. 
Further, $B \cong eAe/eIe$ and $(fAe)(eIe)=fAeJ\subseteq fJ\subseteq fLe$, 
and hence $N$ is a right $B$-module. 
Moreover, $N$ is also a left module over $S=fAf/fLf$ 
and $F/L$ is isomorphic to the triangular matrix algebra
\[
\Lambda=
   \begin{bmatrix}
        S&N\\ 0&B
   \end{bmatrix}
.
\]
Let $X$ be an indecomposable direct summand of $M$.
Assume first that $X$ is not the radical of an indecomposable projective
module in $\mod A$.
Then it follows from Lemma~\ref{lem:3.1} that
we have in $\mod B$ an almost split sequence
\[
   0\longrightarrow X\longrightarrow Y\longrightarrow Z\longrightarrow 0
\]
which is also an almost split sequence in $\mod A$,
and consequently in $\mod F$.
Since $\Lambda$ is a quotient algebra of $F$ and 
$B$ is a quotient algebra of $\Lambda$,
we conclude that it is also an almost split sequence in $\mod \Lambda$.
Applying now \cite[Lemma~5.6]{SY1}
(or \cite[Theorem~VII.10.9]{SY11}) 
we conclude that $\Hom_B(N,X)=0$.
Assume now that $X = \rad P$ for an indecomposable projective
module $P$ in $\mod A$.
Suppose that  $\Hom_B(N,X) \neq 0$.
It follows from the assumption imposed on $\Delta$
that every direct predecessor of $X$ in $\Gamma_A$
lies on $\Delta$ and is not the radical of an indecomposable
projective module in $\mod A$.
Moreover, by Proposition~\ref{prop:3.5},
$\Delta$ is the canonical section $\Delta_T$ 
of the connecting component $\cC_T$ of $\Gamma_B$.
Then $\Hom_B(N,X) \neq 0$ forces that $X$ is an indecomposable
direct summand of $N$.
Recall that $N=fAe/fLe$ with
$L=fAeAf+fJ+fAeAfAe+eAf+eIe$.
Hence we obtain $P = f A$ and $X = \rad P = f A e$.
Then we conclude that $f L e = 0$, and hence $f J = 0$.
But it is not possible because $f(\rad A) = \rad f A$
lies on $\Delta$, and is equal to $f J$.

Summing up, we obtain that $\Hom_B(N,M)=0$.
Since every indecomposable module in $\mod B$
is either generated or cogenerated by $M$,
we conclude that $\Hom_{B}(M,N)\neq0$.
\end{proof}

Applying Lemmas~\ref{lem:3.6}--\ref{lem:3.9}  as in 
\cite[Proposition~5.9]{SY1},
we obtain the following proposition.

\begin{proposition}
\label{prop:3.10}
We have $I e = J$ and $e I = J'$.
\end{proposition}

It follows from 
Lemma~\ref{lem:3.7} and Proposition~\ref{prop:3.10} that
$r_{A}(I)=J'=eI$ and $l_{A}(I)=J$.
Moreover, since $B$ is a tilted algebra,
the valued quiver $Q_B$ of $B$ is acyclic.
Then applying Theorem~\ref{th:2.1}
we conclude that $A$ is socle equivalent to the orbit
algebra $\widehat{B}/(\varphi \nu_{\widehat{B}})$ for some
positive automorphism $\varphi$ of $\widehat{B}$.
Further, if $K$ is an algebraically closed field,
then $A$ is isomorphic to $\widehat{B}/(\varphi \nu_{\widehat{B}})$.

\section{Proof of the sufficiency part of Theorem~\ref{th:main}}

We start with general facts.

Let $\Lambda$ be a self-injective algebra.
Then for any indecomposable projective module $P$
in $\mod \Lambda$ we have a canonical almost split
sequence
\[
  0 \to \rad P \to (\rad P / \soc P) \oplus P 
     \to P / \soc P \to 0 ,
\]
and hence $\rad P$ is a unique direct predecessor of $P$
and $P/\soc P$ is a unique direct successor of $P$ in $\Gamma_{\Lambda}$.
Hence, the Auslander-Reiten quiver $\Gamma_{\Lambda/\soc(\Lambda)}$
is obtained from $\Gamma_{\Lambda}$ by deleting all projective
modules $P$ and the arrows $\rad P \to P$ and $P  \to P / \soc P$.
We also note that if $\Delta$ is a stable slice of $\Gamma_{\Lambda}$
then $\Delta$ is a full valued subquiver of  
$\Gamma_{\Lambda/\soc(\Lambda)}$.
Hence we have the following fact.

\begin{proposition}
\label{prop:4.1}
Let $\Lambda$ and $A$ be two socle equivalent self-injective
algebras and 
$\phi : \mod \Lambda/\soc(\Lambda) \to \mod A/\soc(A)$
the isomorphism of module categories induced by an algebra
isomorphism 
$\varphi : \Lambda/\soc(\Lambda) \to A/\soc(A)$.
Then a full valued subquiver $\Delta$ of $\Gamma_{\Lambda}$
is a hereditary almost right regular slice of $\Gamma_{\Lambda}$
if and only if $\phi(\Lambda)$ is a hereditary almost right regular 
stable slice of $\Gamma_{A}$.
\end{proposition}

Therefore, for proving the sufficiency part of Theorem~\ref{th:main},
we may assume that $A = \widehat{B}/(\varphi \nu_{\widehat{B}})$
for a tilted algebra $B$ and a positive automorphism $\varphi$
of $\widehat{B}$.
We divide that proof into three cases.

\begin{proposition}
\label{prop:4.2}
Assume $A$ is of infinite representation type.
Then $\Gamma_A$ admits a hereditary right regular
stable slice $\Delta$.
\end{proposition}

\begin{proof}
By \cite{Ho,HW} and the assumption, we conclude
that $B$ is not a tilted algebra of Dynkin type.
It follows from general theory (see \cite{ANS,EKS,SY3,SY4})
that  $\Gamma_A$ admits an acyclic component $\cC$
containing a right stable full translation subquiver $\cD$
which is closed under successor in $\cC$ and generalized
standard in the sense of \cite{S2}
(the restriction of the infinite radical $\rad_A^{\infty}$
of $\mod A$ to $\cD$ is zero).
We note that $\cD$ does not contain projective module,
because $\cD$ is right stable.
Then we may choose in $\cD$ a full valued connected
subquiver $\Delta$ which intersects every $\tau_A$-orbit
in $\cD$ exactly once.
Clearly, $\Delta$ is a right regular finite stable slice of $\Gamma_A$.
Moreover, since $\cD$ is generalized standard,
we obtain that $\Delta$ is a hereditary slice of $\Gamma_{A}$.
\end{proof}

\begin{proposition}
\label{prop:4.3}
Let $A$ be a Nakayama algebra.
Then $\Gamma_A$ admits a hereditary almost right regular
slice $\Delta$.
\end{proposition}

\begin{proof}
Let $P$ be an indecomposable projective module in $\mod A$.
Then using the structure of almost split sequence over
Nakayama algebras (see \cite[Theorems I.10.5 and III.8.7]{SY9})
we conclude that there is a sectional path $\Delta$ of the form
\[
  \soc P = X_1 \to X_2 \to \dots \to X_{n-1} \to X_n = \rad P 
\]
such that the $\tau_A$-orbits of these modules exhaust all
indecomposable non-projective modules in $\mod A$.
Moreover, the $\tau_A$-orbit of $\rad P$ consists of the radicals
of all indecomposable projective modules in $\mod A$.
Hence $\Delta$ is an almost right regular stable
slice of $\Gamma_A$.
Since $A = \widehat{B}/(\varphi \nu_{\widehat{B}})$
with $\varphi$ being a positive automorphism of $\widehat{B}$,
we conclude that 
$\rk K_0 (A) \geq \rk K_0 (\T(B)) = \rk K_0 (B)$,
where $\T(B) = B \ltimes D(B) = \widehat{B}/(\nu_{\widehat{B}})$.
On the other hand, it follows from \cite{Ho,HW}
that $\rk K_0 (B)$ is the number of $\tau_A$-orbits in $\Gamma_A^s$.
Therefore, the number of pairwise non-isomorphic indecomposable
projective modules in $\mod A$ is at least $n$.
This implies that $\rad P$ has multiplicity-free composition factors.
But then the endomorphism algebra of the direct sum of modules
on $\Delta$ is a hereditary algebra and $\Delta^{\op}$
is its valued quiver.
Summing up, we conclude that $\Delta$ is a hereditary slice 
of $\Gamma_A$.
\end{proof}

\begin{proposition}
\label{prop:4.4}
Assume that $A$ is of finite representation type
but not a Nakayama algebra.
Then $\Gamma_A$ admits a hereditary right regular
slice $\Delta$.
\end{proposition}

\begin{proof}
We choose a right regular stable slice $\Delta$ of $\Gamma_A$
following the proof of \cite[Theorem~3.1]{JPS}.
Namely, since $A$ is not a Nakayama algebra,
there exists an indecomposable projective module $P$
such that $P/\soc(P)$ is not the radical of a projective module.
Consider the full valued subquiver $\Delta_P$ of $\Gamma_A$
given by $\tau_A^{-1} (P/\soc (P))$ and all indecomposable
modules $X$ such that there is a non-trivial sectional 
path in $\Gamma_A^s$ from $P/\soc(P)$ to $X$.
It is shown in \cite{JPS} that $\Delta_P$ does not 
contain $Q/\soc(Q)$  for any indecomposable projective
module $Q$ in $\mod A$.
Clearly, $\Delta_P$ is a stable slice of $\Gamma_A$.
Then $\Delta = \tau_A (\Delta_P)$ is a right
regular stable slice of $\Gamma_A$.
We claim that $\Delta$ is a hereditary slice.
Let $g = \varphi \nu_{\widehat{B}}$ and $G$ be 
the infinite cyclic group generated by $g$.
Consider the Galois covering
$F : \widehat{B} \to \widehat{B}/G = A$ and the push-down
functor $F_{\lambda} : \mod \widehat{B} \to \mod A$
associated to it.
Since $B$ is tilted of Dynkin type,
$\widehat{B}$ is a locally representation-finite
locally bounded $K$-category (by \cite{Ho,HW})
and hence $F_{\lambda}$ is dense,
preserves almost split sequences, and $\Gamma_A$
is the orbit translation quiver $\Gamma_{\widehat{B}}/G$
of $\Gamma_{\widehat{B}}$ with respect to the induced
action of $G$ on $\Gamma_{\widehat{B}}$ (see \cite[Theorem~3.6]{G}).
Moreover, $F_{\lambda}$ is a Galois covering
of module categories, that is, 
for any indecomposable modules $X$ and $Y$ in $\mod \widehat{B}$,
$F_{\lambda}$ induces an isomorphism of $K$-vector spaces       
\[
  \bigoplus_{r \in \mathbb{Z}} \Hom_{\widehat{B}}(X,g^r Y)
    \to
  \Hom_A\big(F_{\lambda}(X),F_{\lambda}(Y)\big)  .
\]
In particular, we conclude that there exists a right
regular stable slice $\Omega$ of $\Gamma_{\widehat{B}}$
such that $F_{\lambda}(\Omega) = \Delta$.
Let $N$ be the direct sum of all indecomposable 
$\widehat{B}$-modules lying on $\Omega$.
Then $M = F_{\lambda}(N)$ is the direct sum of all 
indecomposable $A$-modules lying on $\Delta$.
Consider the annihilator algebra 
$C = \widehat{B}/\ann_{\widehat{B}}(\Omega)$
of $\Omega$ in $\widehat{B}$.
Since $N$ is a finite-dimensional $\widehat{B}$-module, 
the support of $N$ is finite, and hence 
$C$ is a finite-dimensional $K$-algebra.
Further, because $\Gamma_{\widehat{B}}$ is an acyclic quiver
and $\Omega$ is a right regular stable slice of $\Gamma_{\widehat{B}}$,
we conclude that $\Omega$ is a faithful section of $\Gamma_{C}$
such that $\Hom_C(N, \tau_C N) = 0$.
Then it follows from the
criterion of Liu and Skowro\'nski
\cite[Theorem~VIII.7.7]{SY11} that $N$
is a tilting $C$-module, $H = \End_C(N)$ is a hereditary algebra,
$T = D(N)$ is a tilting module in $\mod H$,
$C \cong \End_H(T)$ canonically,
and $\Omega$ is the section of $\Gamma_C$
determined by $T$.
In particular, $C$ is a tilted algebra of Dynkin type
$\Omega^{\op} = \Delta^{\op}$.
We note also that ${\widehat{B}} = {\widehat{C}}$,
and hence $\nu_{\widehat{B}} = \nu_{\widehat{C}}$ (see \cite{HW}).
Since $g = \varphi \nu_{\widehat{B}}$ with $\varphi$
a positive automorphism of ${\widehat{B}}$,
we conclude that, for any integer $r$, 
the categories $C$ and $g^r(C)$ have no common objects,
and consequently $\Hom_{\widehat{B}}(N,g^r N) = 0$.
Then we obtain isomorphisms of $K$-vector spaces
\[
  H = \End_C(N) 
  = \End_{\widehat{B}}(N)   
  = \bigoplus_{r \in \mathbb{Z}} \Hom_{\widehat{B}}(N,g^r N)  
  \cong \End_A\big(F_{\lambda}(N)\big)   
  = \End_A(M)   
.
\]
Hence $\End_A(M)$ is a hereditary algebra and $\Delta^{\op}$
is its valued quiver. 
Therefore, $\Delta$ is a hereditary stable slice of $\Gamma_A$.
\end{proof}

We end this section with an example illustrating the above
considerations.

\begin{example}
\label{ex:4.5}
Let $Q$ be the quiver
\[
    \xymatrix{
      1 \ar@<+.5ex>[r]^{\alpha} & \ar@<+.5ex>[l]^{\beta}  
      3 \ar@<+.5ex>[r]^{\gamma} & \ar@<+.5ex>[l]^{\sigma}  
      2
    } ,
\]
$R$ the ideal in the path algebra $K Q$ of $Q$ over $K$
generated by 
$\beta\alpha-\gamma\sigma$,
$\alpha \beta$ and $\sigma \gamma$,
and $A = K Q/R$.
Moreover, let $Q^*$ be the quiver
\[
    \xymatrix{
      1 \ar[r]^{\alpha}  & 3 & \ar[l]_{\sigma}  2
    } 
\]
and $B = K Q^*$ the associated path algebra.
Then $A$ is the self-injective algebra of the form
$\widehat{B}/(\varphi \nu_{\widehat{B}})$
where $\varphi$ is the positive automorphism of 
$\widehat{B}$ given by
\begin{align*}
 &&
 \varphi(e_{m,1}) &= e_{m,2}, &
 \varphi(e_{m,2}) &= e_{m,1}, &
 \varphi(e_{m,3}) &= e_{m,3}, &
 \mbox{for all } m  \in \mathbb{Z} .
 &&
\end{align*} 
For each $i \in \{1,2,3\}$, we denote by $P_i$ and $S_i$
the indecomposable projective module and simple
module in $\mod A$ associated to the vertex $i$.
Then the Auslander-Reiten quiver $\Gamma_A$ of $A$
is of the form
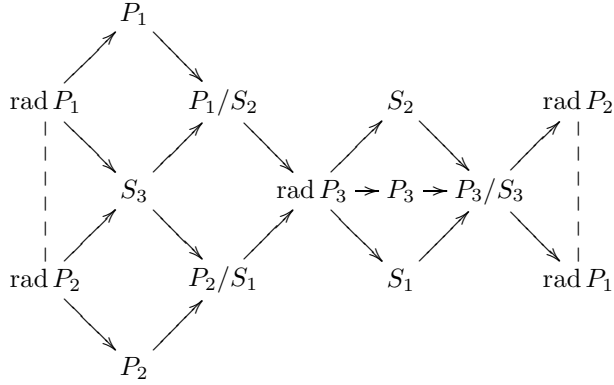
\[
\begin{xy}
0;/r.28pc/:
(20,40)*+{P_1}="0,2" ;
(10,30)*+{\rad P_1}="1,1" ;
(30,30)*+{P_1/S_2}="1,3" ;
(50,30)*+{S_2}="1,5" ;
(70,30)*+{\rad P_2}="1,7" ;
(20,20)*+{S_3}="2,2" ;
(40,20)*+{\rad P_3}="2,4" ;
(50,20)*+{P_3}="2,5" ;
(60,20)*+{P_3/S_3}="2,6" ;
(10,10)*+{\rad P_2}="3,1" ;
(30,10)*+{P_2 / S_1}="3,3" ;
(50,10)*+{S_1}="3,5" ;
(70,10)*+{\rad P_1}="3,7" ;
(20,0)*+{P_2}="4,2" ;
\ar @{->} "1,1";"0,2"
\ar @{->} "0,2";"1,3"
\ar @{->} "1,1";"2,2"
\ar @{->} "1,3";"2,4"
\ar @{->} "1,5";"2,6"
\ar @{->} "2,2";"1,3"
\ar @{->} "2,4";"1,5"
\ar @{->} "2,6";"1,7"
\ar @{->} "2,4";"2,5"
\ar @{->} "2,5";"2,6"
\ar @{->} "2,2";"3,3"
\ar @{->} "2,4";"3,5"
\ar @{->} "2,6";"3,7"
\ar @{->} "3,1";"2,2"
\ar @{->} "3,3";"2,4"
\ar @{->} "3,5";"2,6"
\ar @{->} "3,1";"4,2"
\ar @{->} "4,2";"3,3"
\ar @{--} "1,1";"3,1"
\ar @{--} "1,7";"3,7"
\end{xy}
\]
Then we have the hereditary right stable slices associated to
$P_1,P_2,P_3$ (as in the proof of Proposition~\ref{prop:4.4})
\begin{align*}
 &&
 \tau_A(\Delta_{P_1}) = \tau_A(\Delta_{P_2}) : 
 &&&
 P_1 / S_2 \leftarrow S_3 \to P_2/S_1
 &&
 \\
 &&
 \tau_A(\Delta_{P_3}) : 
 &&&
 S_2 \to P_3/S_3 \leftarrow S_1 .
 &&
\end{align*} 
On the other hand, we have the stable slices of $\Gamma_A$
\begin{align*}
 &&
 \rad P_1 \to S_3 \to P_2/S_1
 &&
 \mbox{and}
 &&
 \rad P_2 \to S_3 \to P_1/S_2
 &&
\end{align*} 
which are not hereditary.
\end{example}

\section*{Acknowledgements}
\label{ackref}
The research described in this paper was completed during the visit
of the second named author at Nicolaus Copernicus University 
in Toru\'n (October 2017).
The authors were supported by the research grant 
DEC-2011/02/A/ST1/00216 of the National Science Center Poland.
The second named author was also supported by
JSPS KAKENHI Grant Number 25400036.

\end{document}